\documentclass[12pt, reqno]{amsart}
\usepackage{fullpage}
\usepackage{amssymb,amsmath,mathrsfs,amsthm}
\usepackage{amsfonts}
\usepackage{enumitem}
\usepackage{mathtools}
\usepackage{caption}
\usepackage{esint,comment}
\usepackage[dvips]{graphicx}
\usepackage{xcolor,scalerel} 


\allowdisplaybreaks
\newtheorem{thm}{Theorem}

 \newtheorem{lemma}[thm]{Lemma}
\newtheorem{prop}[thm]{Proposition}

\theoremstyle{definition}

\def \no#1#2#3 {{\bf #1} (#3), #2.}
\def \eds#1#2#3 {#1, #2, #3.}

\def\R{{\mathbb R}}

\def\d{{\rm d}}

\def\T{{\mathbb T}}

\def\:{{\colon}}

\def\be#1{\begin{equation}\label{#1}}
\def\ee{\end{equation}}

\def\<{\langle}
\def\>{\rangle}
\def\coloneqq{:=}

\newcommand{\na}{\nabla}

\newcommand{\lec}{\lesssim}
\newcommand{\bs}{\begin{split}}
\newcommand{\essss}{\end{split}}

\renewcommand{\lec}{\lesssim}

\renewcommand{\div}{\operatorname{div}}
 
\newcommand{\eqnb}{\begin{equation}}
\newcommand{\eqne}{\end{equation}}

\newcommand{\ii}{\mathrm{i}}

\renewcommand{\ee}{\mathrm{e}}

\newcommand{\p}{\partial}

\newcommand{\re}{\mathrm{Re}}

\newcommand{\vv }{{v}}
\renewcommand{\aa }{{\alpha }}
\newcommand{\ww }{{w}}
\newcommand{\bb }{{\beta }}

\renewcommand{\T}{\mathbb{T}}
\renewcommand{\R}{\mathbb{R}}
\newcommand{\RR}{\mathbb{R}}

\newcommand{\Z}{\mathbb{Z}}

\renewcommand{\d}{\mathrm{d}}

\newcommand{\UB}{B^{d/p-2}}
\newcommand{\B}{B^{d/p-1}}
\newcommand{\BB}{B^{d/p}}
\newcommand{\BBB}{B^{d/p+1}}

\begin{document}

\title{Global-in-time stability of ground states of a pressureless hydrodynamic model of collective behaviour}

\author{Piotr B. Mucha, Wojciech S. O\.za\'{n}ski}
\address{
\begin{minipage}{\linewidth}
Piotr B. Mucha, \\
Institute of Applied Mathematics and Mechanics, University of Warsaw, 02-097 Warsaw, Poland\\\texttt{p.mucha@mimuw.edu.pl}\\  \\
Wojciech S. O\.za\'{n}ski\\
Department of Mathematics, University of Southern California, Los Angeles, CA 90089\\
\texttt{ozanski@usc.edu}
\end{minipage}
} 

\begin{abstract} 
We consider a pressureless hydrodynamic model of collective behaviour, which is concerned with a density function $\rho$ and a velocity field $v$ on the torus, and is described by the continuity equation for $\rho$, $\p_t \rho + \div (v\rho )=0$, and a compressible hydrodynamic equation for $v$, $\rho v_t + \rho v\cdot \nabla v - \Delta v = -\rho \nabla K  \rho$ with a forcing modelling collective behaviour related to the density $\rho$, 
where $K$ stands for the interaction potential, defined as the solution to the Poisson equation on $\T^d$. We show global-in-time stability of the ground state $(\rho , v)=(1,0)$ if the perturbation $(\rho_0-1 ,v_0)$ satisfies $\| v_0 \|_{B^{d/p-1}_{p,1}(\T^d )} + \| \rho_0-1 \|_{B^{d/p}_{p,1}(\T^d )} \leq \epsilon$ for sufficiently small $\epsilon >0$. 
\end{abstract}
\maketitle

\section{Introduction}

The subject of this note is a model of motion of a pressureless gas driven by the following set of laws:
\eqnb\label{pressureless}
\begin{split}
\p_t \rho + \mathrm{div}\, (v\rho ) =0 ,& \qquad \\
\rho v_t + \rho v\cdot \nabla v - \Delta v = -\rho \nabla K  \rho & \qquad \mbox{ in } [0,\infty ) \times \T^d,
\end{split}
\eqne
considered, for simplicity, on the $d$-dimensional torus $\T^d=\R^d / (2\pi \Z )^d $. The first equation is the mass conservation, prescribing the dynamics of the density $\rho$ under the flow $v$, and the second one is the momentum equation.  Here $K=(-\Delta)^{-1}$ is the operator such that $\Psi \coloneqq K  \rho$ satisfies 
\eqnb\label{def_of_K_operator}
-\Delta \Psi = \rho - \{ \rho \}, \mbox{ \ where \ } \{ \rho \} \coloneqq \displaystyle \int \rho\, \d x
\eqne
and $\int \Psi =0$.  Here and below we use the short-hand notation $\int \equiv \int_{\T^d }$, and we will often omit ``$\d x$'', for brevity. 
We note that, since the average $\{ \rho \}$ is preserved by the flow, we assume, without loss of generality, that $\{ \rho \} =1$ for all times.

\smallskip 

Model \eqref{pressureless} arises as a nonlinear repulsion model of collective behaviour \cite{cwz_2018,ccp_2017,ha-tad}.
An exclusive feature of this type of systems is the lack of the internal force, represented in the 
classical mechanics by the pressure. Instead of it we consider an external force given by a repulsion of electromagnetic type of the from of the Poisson potential $K=(-\Delta)^{-1}$. If we assumed the presence of the pressure function (of type $\nabla p(\rho)$ for instance), then we would have obtained a variant of the compressible Navier-Stokes system. 

 In fact, there are a number of results for models with the pressure \cite{cwz_2018,ccp_2017} of type $p \sim \rho^m$. The case $m=0$ is not well-understood yet, which is partially our motivation to study system \eqref{pressureless}.  
Moreover, \eqref{pressureless} is also related to the models of self-gravitational gases \cite{AB,DLY,Mper}. However, the closest, in the authors' opinion, is a result regarding the pressureless Euler-Poisson system \cite{hadzic_jang}, see also \cite{CCZ}. The latter result is merely mono-dimensional and, instead of dissipation, a friction term is taken into account. What is interesting from the mathematical viewpoint is that the friction 
somehow gives better stability properties that dissipation, even in the case of a bounded domain. 
In contrast, for system \eqref{pressureless} we do obtain global-in-time existence of unique solutions for small data, but any exponential decay as $t\to \infty $ can not be expected.
This is a consequence of the fact that the spectrum of an operator coming from the linearization of (\ref{pressureless}) is not cut from zero, see \eqref{real_parts_lambda_k_s} below for details.
Another related result is concerned with a detailed analysis of the
mono-dimensional Euler-Poisson system \cite{ELT}, where many cases have been discussed. More dimensional cases require some modifications, see \cite{LT1}.
In the case $K\equiv 0$ we refer the reader to \cite{DMT}, where the system is analyzed in a nonstandard framework of the Lorentz spaces and time-weighted norms.

From the mathematical viewpoint, the key difficulty of \eqref{pressureless} is the lack of the effective viscous flux which relates the divergence of the velocity $v$ with the pressure $p(\rho)$ in the form
\begin{equation*}
    \div v - p(\rho).
\end{equation*}
 This quantity is often used in the theory of compressible Navier-Stokes system to smooth out the density, by proving its decay or integrability in time, which in turns gives enough compactness to yield existence of  weak solutions \cite{Feir,Lions}. 
In the case of \eqref{pressureless}, there is no such simple quantity with fine properties.
This is one of the reasons why the general analysis of systems of type (\ref{pressureless}) is at the borderline of the modern PDE techniques.\\

In order to analyze \eqref{pressureless} one could consider the quasi-stationary approximation  of (\ref{pressureless}), which leads to the following aggregation type equation
\begin{equation}\label{eq:agg}
    \partial_t \rho - (-\Delta)^{-1} \div (\rho \nabla K  \rho)=0.
\end{equation}
An analysis of the above system could deliver  possible static solutions to (\ref{pressureless}) and is related to the issue of stability. 

On the other hand, considering (\ref{pressureless}) from the viewpoint of the energy, in analogy to the compressible Navier-Stokes equation, we can test the momentum equation with $v$ to get
\begin{equation}\label{stability_heu}
    \frac12 \frac{\d}{\d t} \int (\rho |v|^2 + \rho K  \rho ) \d x 
    +\mu \int |\nabla v|^2 \d x =0.
\end{equation}
In order to understand the meaning of the  term involving $K$ we note that $\Psi \coloneqq K\rho = K(\rho -1)$ has zero average, which implies that 
\[
    \int \rho K  \rho\,  \d x = \int (\rho -1) K (\rho -1) \, \d x = \int (-\Delta \Psi) \Psi \, \d x      = \int |\nabla \Psi|^2\,  \d x = \|\rho -1 \|_{H^{-1}}^2.
\]

This shows that the energy $ \int \rho |v|^2 \d x + \|\rho -1\|^2_{H^{-1}}$ decreases in time, which suggests that system (\ref{pressureless}) is stable, at least around static solutions. It can be interpreted as the structure of the force $-\rho \nabla K  \rho$, which says that particles repel each other. This is related to the phenomenon of the electron gas \cite{Germain}. In order to retain positivity of the density $\rho$ for all times, we consider the case of a bounded domain.

The purpose of this note is to establish the first stability result of the ground states of \eqref{pressureless}. To this end, we  focus on the case of the torus $\T^d$, for the sake of simplicity. In such case the ground state of \eqref{pressureless} is $(\rho , v)=(1,0)$, and we prove global-in-time stability of this state. 

In order to obtain such stability result we  will make use of some tools from the theory of regular solutions for the compressible Navier-Stokes system. The first approach to such system is based on the $L^2$ setting \cite{Mat-Ni}. However we will extend the techniques from \cite{Mu1,MZaj1,MZaj2}, developed for the $L^p$
spaces, as well as from \cite{Dan1}, which focuses on the Besov space setting on the whole space $\R^3$. These methods can be further developed to yield the following.

\begin{thm}[Main result]\label{thm_main}
Given $d \geq 3$, $p\in (\min(d/2,2), d)$ there exists $\epsilon >0$ with the following property. For every $\rho_0-1 \in B^{d/p}_{p,1}(\T^d)$ and $v_0 \in B^{d/p-1}_{p,1}(\T^d)$ such that $\int \rho_0 v_0 =0$ and 
\[
\| v_0 \|_{B^{d/p-1}_{p,1}(\T^d)} + \| \rho_0-1 \|_{B^{d/p}_{p,1}(\T^d) } \leq \epsilon
\]
there exists a unique global in time solution $(\rho , v)$ of \eqref{pressureless}, such that
$$
\rho -1 \in C_b([0,\infty);B^{d/p}_{p,1}(\T^d)), \quad
v_t, \nabla^2 v \in L^1(0,\infty;B^{d/p-1}_{p,1}(\T^d)),
$$
with
\[
\| \rho -1 \|_{L^\infty \left((0,\infty ); B^{d/p}_{p,1} (\T^d)\right)} + \| v_t \|_{L^1 \left((0,\infty ); B^{d/p-1}_{p,1} (\T^d)\right)} + \| v \|_{L^1 \left( (0,\infty ); B^{d/p+1}_{p,1} (\T^d)\right) } \leq C \epsilon ,
\]
where $C=C(d,p)>1$ is a constant.
\end{thm}
Here $C_b (I; X)$ denotes the space of continuous and bounded functions from interval $I$ to a Banach space $X$. Here $B^s_{p,1} (\T^d )$ stand for the Besov space on the torus (see Section~\ref{sec_prelims} for details). In order to motivate this functional framework we first observe that the transport equation for $\rho$ gives the a priori bound
\eqnb\label{necessary}
\| \rho (t) \|_{L^\infty } \leq \| \rho_0 \|_{L^\infty } \exp \left( \int_0^t \| \div v(s) \|_{L^\infty }  \d s\right).
\eqne
This suggests that the condition
\eqnb\label{suggestion}
\div v \in L^1((0,\infty );L^\infty)
\eqne 
is necessary for any global well-posedness result. On the other hand, in order to construct the solution claimed by the above theorem, one would consider a linearization of \eqref{pressureless} around the ground state. In order to effectively analyze such linearization one would need to use maximal regularity of the heat equation $\p_t u - \Delta u =f $ in a space of the form $L^1((0,\infty ); X)$ where $X$ is some Banach space. However, it is well-known that if $X$ is reflexive (more precisely $UMD$,
see \cite{DHP}) then such maximal regularity holds  in $L^q(0,T; X)$ only for $1<q<\infty$. Thus, in order to reach the borderline case $q=1$, we need to find a non-reflexive Banach space $X$. This suggests that we should leave the classical $L^p$ framework and enter the universe of the Besov spaces. It naturally leads us to consider the spaces of the form $L^1((0,\infty); B^s_{p,1})$, considered by \cite{Dan-Mu,Dan-Mu2}. In fact, maximal regularity of the heat equation holds in $L^1 ((0,\infty ); B^s_{p,1})$ (see \eqref{max_reg} below), which is one of the most significant property of Besov spaces with the last index one. Moreover, $B_{p,1}^s \subset L^\infty$ for $s\geq d/p$, which, in light of \eqref{suggestion}, suggests that we should consider
\[
\na v \in L^1 \left( (0,\infty ); B^{d/p}_{p,1} \right),
\]
which naturally leads us to the functional setting considered in Theorem~\ref{thm_main}. 

Thanks to this choice of functional setting, the hyperbolic character of the continuity equation for $\rho$ can be removed. Moreover, due to the smallness assumption of Theorem~\ref{thm_main}, we expect that 
\[
\left\| \int_0^t |\nabla v (s,x ) |\d s \right\|_{L^\infty } <\frac12,
\]
which suggests that the Lagrangian coordinates should be well defined. In fact, the problem \eqref{pressureless} becomes simpler in such coordinates. In order to describe the main difficulties, we first introduce the Lagrangian setting.

Let $X(t,y) $ be the solution of the system
\[
 \frac{dX(t,y)}{dt}=v(t,X(t,y)), \qquad X|_{t=0}=y. 
\]
The Lagrangian coordinates $y$ are given by the relation
\begin{equation}\label{lagr_cord_int_form}
 X(t,y)\coloneqq y+\int_0^t v(\tau ,X(\tau ,y)) \d \tau .
\end{equation}
We set
\begin{equation}
 \eta(t,y)\coloneqq \rho(t,X(t,y)), \qquad\qquad u(t,y)\coloneqq v(t,X(t,y)).
\end{equation}
The transformation matrix reads
\begin{equation}\label{def_A}
A\coloneqq  \left( \frac{dX}{dy} \right)^{-1}=\left( I+\int_0^t \nabla u  \right)^{-1}.
\end{equation}
Since we assumed the average of the density is one, and we aim at analysis of the flow around this state, we introduce $a$ as follows
\begin{equation}\label{take_eta}
 \eta=1+a.
\end{equation}
The equations \eqref{pressureless} in Lagrangian coordinates become
\begin{equation}\label{eqs_lagrangian1}
 \begin{array}{lr}
  \eta_t + \eta \div_u u=0, & \\[7pt]
  \eta u_t -\Delta_u u=-\eta  \na_u (-\Delta_u )^{-1} a & \qquad \mbox{in \ } [0,\infty )\times \T^d, 
 \end{array}
\end{equation}
with initial conditions $ \eta|_{t=0}=\rho_0$, $ u|_{t=0}=v_0$, where we have denoted by $\na_u $, $\div_u$ and $\Delta_u$ the gradient and the Laplace operator, respectively, with respect to the $x$ variable, that is in the Eulerian setting. For example 
$$
\div_u u (t,y) = \left[ \div v (t,x) \right]_{x=X(t,y)}, \mbox{ \ and  \ } 
\Delta_u u (t,y) = \left[ \Delta v(t,x) \right]_{x=X(t,y)}.
$$
We will use the notation $\div$, $\na$ and $\Delta$ to denote the respective differential operators of $u$, $a$ or $\eta$ with respect to their spatial variable, namely $y$.
Similarly, we denote by $(-\Delta )^{-1}$ the solution operator of \eqref{def_of_K_operator} in variable $y$, and we have also denoted by $(-\Delta_u )^{-1} g(y)$ the solution operator in the Eulerian variables. Namely, since $a(t,y) = \left[ \rho (t,x) -1 \right]_{x=X(t,y)}$ we have
\[
\na_u (-\Delta_u )^{-1} a = \na_u \left[ (-\Delta )^{-1} (\rho -1) \right]_{x=X(t,y)} .
\]

We will show (in Section~\ref{sec_equiv} below) that Theorem~\ref{thm_main} is  equivalent to the following restatement in the Lagrangian coordinates. 

\begin{prop}[Main result in the Lagrangian coordinates]\label{prop_main}
Given $d \geq 3$, $p\in (\min(d/2,2), d)$ there exists $\epsilon >0$ with the following property. For every $a_0 \in B^{d/p}_{p,1}(\T^d)$ and 
$u_0 \in B^{d/p-1}_{p,1}(\T^d)$ with 
\[
 \| a_0 \|_{B^{d/p}_{p,1}(\T^d) } +\| u_0 \|_{B^{d/p-1}_{p,1}(\T^d)} \leq \epsilon,
\]
there exists a unique Lagrangian map $X$  and a pair $(a  , u)$ such that $X(t) -\mathrm{id} = \int_0^t u (\tau ) \d \tau $ holds for $t>0$, the equations \eqref{eqs_lagrangian1} hold on $(0,\infty ) \times \T^d$ with initial conditions $(a_0, u_0)$, and 
\[\begin{split}
&\| a \|_{L^\infty \left(  (0,\infty ); B^{d/p}_{p,1} (\T^d ) \right)} +  \| a_t \|_{L^1 \left(  (0,\infty ); B^{d/p}_{p,1} (\T^d ) \right)}   +  \| a - \{ a \} \|_{L^1 \left(  (0,\infty ); B^{d/p-2}_{p,1} (\T^d ) \right)}\\
&\hspace{1cm}+ \| u \|_{L^\infty \left(  (0,\infty ); B^{d/p-1}_{p,1} (\T^d ) \right)} +  \| u_t \|_{L^1 \left(  (0,\infty ); B^{d/p-1}_{p,1} (\T^d ) \right)} + \| u-\{ u \}  \|_{L^1 \left(  (0,\infty ); B^{d/p+1}_{p,1} (\T^d ) \right)}\\ 
&\hspace{1cm}+\| \na X- I  \|_{L^\infty \left(  (0,\infty ); B^{d/p}_{p,1} (\T^d ) \right)}\leq C \epsilon,
\end{split}
\]
where $C=C(d,p)>1$ is a constant.
\end{prop}  
We note that in Proposition~\ref{prop_main} we only obtain the smallness the averages $\{ a \}, \{ u \}$ of $a,u$ that is uniform in time, while in the Eulerian coordinates (i.e.~in Theorem~\ref{thm_main}) we obtain $L^1$ control, which is a consequence of the conservation of mass $\int \rho$ and the conservation of momentum $\int \rho v$, see  \eqref{cons_prop} for details.

  We note that the claim of Proposition~\ref{prop_main} implies in particular that 
\[
    \nabla u \in L^1(\R_+;L^\infty) \mbox{ \ \ and \ \ }
    a \in L^\infty (\T^d \times \R_+).
\]

In order to prove Proposition~\ref{prop_main} we consider the linearization of \eqref{eqs_lagrangian1} the system around the ground state $(a,u)=(0,0)$, see \eqref{lin_sys} below. We establish well-posedness of the linearization in Lemma~\ref{lem_lin_system}. To this end, we apply the explicit formula for the linear system, and we use a multiplier theorem, which is a version of the 1939 Marcinkiewicz theorem \cite{Mar}. We find the maximal regularity estimate, which then determines the regularity framework  used in Proposition~\ref{prop_main}.

In order to consider the nonlinear problem (\ref{eqs_lagrangian1}), we note that the main difficulty of Proposition~\ref{prop_main} is the appearance of the inverse Laplacian $(-\Delta_u )^{-1}$ in the Eulerian coordinates in \eqref{eqs_lagrangian1}. This term, i.e. $\na_u (-\Delta_u )^{-1}a$ needs to be estimated in Besov spaces in Lagrangian coordinates, and this can be achieved by showing that elliptic estimates are stable with respect to the Lagrangian mapping $y \mapsto X(t,y)$, given smallness of the initial data, see \eqref{Delta_commutator} and \eqref{cubic_term} for example. Using this trick we prove Proposition~\ref{prop_main} by linearizing \eqref{eqs_lagrangian1} and then applying Banach Contraction Theorem to obtain a unique global-in-time solution for small data. 

We emphasize that the use of the Besov spaces seems irreplaceable, since, although we are working in a bounded domain, we are not able to obtain any exponential time decay. In fact, as mentioned above, the real spectrum of the linearized system \eqref{lin_sys} is not separated from zero, see \eqref{real_parts_lambda_k_s}.

\smallskip 

Finally we discuss possible directions coming from our result. Firstly, it seems possible to extend this analysis to other operators $K$.  Here we consider a very particular one, but  there is a zoo of other interesting and physically relevant examples, see \cite{CHV,Bed,BFH,CHM} for example. The next problem is to consider the case of the whole space $\R^3$, which seems to be more natural for problem arising from collective behaviors.
Such setting requires a more subtle functional setting crossing the standard definition of Besov spaces, which is related to a number of mathematical challenges which remain to be addressed. One of them is a natural assumption of finite mass $\int_{\R^d} \rho <\infty$, which implies decay of $\rho $ at spatial infinity. It is not clear what decay of $\rho$ should be assumed, but assuming compact support of $\rho$, the system suffers an elliptic degeneration, namely the term $\rho v_t$ just disappears.


We also note that even the change of the sign of $K$ in (\ref{pressureless}) results in an unstable system (recall \eqref{stability_heu}, see also \eqref{lambda_roots} below). In that case it seems natural to expect the density to converge to a single point, namely to a Dirac delta. 

\smallskip 

The structure of the note is as follows. In the next section we introduce the notion of Besov spaces on the torus, ${B}^s_{p,q} (\T^d)$. We discuss some basic properties of such spaces, including the Nikol'skij inequality \eqref{nikolskij1}, embeddings \eqref{embeddings}, as well as maximal regularity of the heat equation \eqref{max_reg}, multiplier properties \eqref{der_in_out}, product laws \eqref{algebra_prop1} and diffeomorphism invariance \eqref{diffeo_prop}. Some of the results we could not find in the literature, and so we provide proofs for the sake of completeness. In Section~\ref{sec_pf_thm} we first discuss a well-posedness result (Lemma~\ref{lem_lin_system}) of the linearization of the equations in the Lagrangian form  \eqref{eqs_lagrangian1} and then prove Proposition~\ref{prop_main}. Section~\ref{sec_equiv} is devoted to the proof of the claimed equivalence of the Eulerian setting (Theorem~\ref{thm_main}) and the Lagragian setting (Proposition~\ref{prop_main}).

\section{Preliminaries}\label{sec_prelims}
We denote by $A \lec B$ the inequality $A\lec C B$, where $C>0$ is a universal constant. If $C$ depends of some parameters, we denote those using  subscripts. By ``$\sim$'' we mean ``$\lec $ and $\gtrsim $''.
We will use the standard notation of the Sobolev space by $H^\beta \coloneqq  H^\beta (\R^d ) \equiv W^{\beta ,2} (\R^d)$. We will also write $L^p\equiv L^p (\T^d)$, and $\| \cdot \|_p \equiv \| \cdot \|_{L^p}$.

Given $u\colon \T^d \to \R$ and $k\in \Z^d$ we denote its $k$-th Fourier mode by $u_k \coloneqq \int_{\T^d } u(x) \ee^{\ii k\cdot x} \d x$.

Let $M\coloneqq \R^d \to \R$ be such that $M\in H^\beta $ for some $\beta > d/2$. Letting $\Lambda \subset \Z^d $ be a finite set, and letting $d_\Lambda \coloneqq \max_{k,l\in \Lambda } |k-l |$, we recall a Fourier multiplier inequality
\eqnb\label{multi_on_torus}
\left\| \sum_{k\in \Lambda } M(k) u_k \ee^{\ii k \cdot x} \right\|_p \leq C \| M_{d_\Lambda } \|_{H^\beta } \| u \|_p, \qquad p\in [1,\infty ], \beta >d/2,
\eqne
where $C=C(d,p,\beta )$, and $M_\lambda \coloneqq M (\lambda \cdot )$ denotes the $\lambda$-dilation of $M$. We refer the reader to Section~3.3.4 in \cite{st_book} for a proof of \eqref{multi_on_torus}. We recall the Nikol'skij inequality,
\eqnb\label{nikolskij}
\| f \|_q \lec_{p,q} d_{\Lambda }^{ \frac{d}{p} - \frac{d}{q} } \| f \|_p
\eqne
for $p,q\in [1,\infty ]$ such that $q\geq p$ and for $f\in L^p$ such that $f_k=0$ for $k\not \in \Lambda $, see Section~3.3.2 in \cite{st_book} for a proof. 

In order to define Besov spaces $B^s_{p,q}( \T^d)$ we first let $h \in C^\infty (R ; [0,1])$ be such that $\chi (z) =1$ for $z\leq 1 $ and $\chi (z) =0$ for $z\geq 2$, and we set
\[\phi_1 ( x) \coloneqq \chi ( |x|/2 ) - \chi (|x|), \qquad \phi_j  \coloneqq \phi_1 (2^{-j} \cdot  )\quad \text{ for } j\geq 2,\qquad \phi_{0 }(x) \coloneqq \chi (|x| ) 
\]
 for $x\in \R^d$. We also set
\[ \psi_j  \coloneqq \phi_1 (2^{-j} \cdot  )\quad \text{ and } \quad \psi_{j\pm k} \coloneqq \sum_{l=j-k}^{j+k} \psi_l \text{ for }j\in \Z, k\geq 0 
\]
For $m\geq 0$ let
\[
P_m u \coloneqq \sum_{k\in \Z^d } \phi_m (k ) u_k \ee^{\ii k\cdot x}.
\]
Note that
\eqnb\label{nikolskij1}
\| P_m u \|_q \lec_{p,q} 2^{\frac{md}{p} - \frac{md}{q}} \| P_m u \|_p,
\eqne
by the Nikol'skij inequality \eqref{nikolskij}.
Given $p,q\in [1,\infty ]$, $s\in \R$ we let $B_{p,q}^s$ denote the ...
\[
\| u \|_{B_{p,q}^s }^q \coloneqq \sum_{m\geq 0 } 2^{smq} \left\| P_m u \right\|_p^q 
\]
recall \cite[Definition~3.5.1(i)]{st_book}. In this work we will be only concerned with functions with vanishing mean, i.e. $\int_{\T^d} f =0$, for which the above sum can be taken over $m\geq 1$.

Note that $B_{p,q}^s$ is a Banach space by Theorem~1 in Section~3.5.1 in \cite{st_book}. We recall the embedding
\eqnb\label{embeddings}
B^{d/p+\delta }_{p,1} \subset B^{d/p}_{p,1} \subset C^0 (\T^d )
\eqne
for every $\delta > 0$, $p\in [1,\infty ]$, see \cite[p.~170]{st_book}. 

Suppose that $f,g \in L^1_{loc} ((0,\infty ); L^1 )$ are such that $\int_{\T^d} f(t) =\int_{\T^d} g(t) =0$ for each $t>0$ and that  
\eqnb\label{par_eq}
\p_t f - \Delta f =g
\eqne
hold in the sense of distributions in $\T^d\times \R_+$. Then 
\eqnb\label{max_reg}
\| f_t \|_{L^1 B^{s}_{p,1}} + \| f \|_{L^1 B^{s+2}_{p,1}} \lec \| g \|_{L^1 B^{s}_{p,1}}
\eqne
for every $p\in [1,\infty ]$, $s\in \R$. In order to verify \eqref{max_reg} we first note that the solution $f$ of \eqref{par_eq} can be characterized in terms of its Fourier coefficients,
\eqnb\label{par_eq_sol}
f_k (t) = \int_0^t \ee^{-k^2 (t-s) } g_k (s) \d s
\eqne
for every $k\in \Z^d \setminus \{0 \}$. Secondly, for every $\alpha >0$
\eqnb\label{each_mode_exp}
\left\| P_m \left( \sum_{k\in \Z^d } \ee^{-\alpha k^2 } g_k \ee^{\ii k \cdot x} \right) \right\|_p \lec_p \ee^{-c\alpha 2^{2m}} \| P_m g \|_p,
\eqne
where $c>0$ is a constant.

Let $m\geq 1$. We take $u\coloneqq P_m g$, $M(\xi ) \coloneqq \ee^{-\alpha |\xi |^2 } \phi_{m\pm 1}(\xi )$. We take $N\coloneqq [d/2]+1$, $\Lambda \coloneqq \{ k\in \Z^d \colon 2^{m-1} \leq |k| \leq 2^{m+1} \}$. We have $d_\Lambda \sim 2^m$, which gives that
\eqnb\label{mult_temp}
\begin{split}
\| M_{d_\Lambda } \|_{H^N}^2 &= \sum_{|\gamma |\leq N } d_{\Lambda }^{2|\gamma |}  \int |  D^\gamma  M ( d_{\Lambda } \xi ) |^2 \d \xi \lec \sum_{|\gamma |\leq N } 2^{(2|\gamma  |-d)m}  \int |  D^\alpha M  |^2  \\
&\lec  \sum_{|\gamma |\leq N } 2^{(2|\gamma  |-d)m}  \int_{ \{ 2^{m-2} \leq |\xi | \leq 2^{m+2} \} } \left( (2^{-m}+ | \alpha \xi | + \ldots + | \alpha \xi |^{ |\gamma | }) \ee^{-\alpha |\xi |^2 } \right)^2 \d \xi   \\
&\lec  \ee^{-\alpha 2^{2m-4}} \int \left( Q_N (\alpha |\xi |^2 )  \ee^{-\frac{\alpha |\xi |^2}{2} } \right)^2 \d \xi  \lec  \ee^{-\alpha 2^{2m-4}},
\end{split}
\eqne
where $Q_N$ denotes a polynomial of order $N$ and, in the second inequality, we obtained the term ``$2^{-m}$'' in the case when all derivatives fall onto $\phi_{m\pm 1}$. If $k$ derivatives fall on ``$\ee^{-\alpha |\xi |^2 }$'' we obtain ``$|\alpha \xi |^k$'', and then each of the other $|\gamma | - k$ derivatives give factors of $2^{-m}$, as $|\xi | \sim 2^m$ and $m\geq 1$. In the third inequality above we also used our choice of $N$, which implies that $2|\gamma |-d \leq 2$. 
Applying the multiplier inequality \eqref{multi_on_torus} gives \eqref{each_mode_exp}, as required.

Multiplying \eqref{par_eq_sol} by $\phi_m \ee^{\ii k\cdot x } $, summing in $k\in \Z^d$ and taking the $L^p$ norm we obtain
\[
\| P_m f(t) \|_p \leq \int_0^t \left\| P_m \left( \sum_{k\in \Z^d } \ee^{-(t-s) k^2 } g_k (s) \ee^{\ii k \cdot x} \right) \right\|_p  \d s \lec \int_0^t \ee^{-c(t-s) 2^{2m}} \| P_m g(s) \|_p \d s
\]
where we used \eqref{each_mode_exp} with $\alpha \coloneqq t-s$ in the last step. Integration over $t\in (0,\infty )$ and using Young's inequality $\| f\ast g \|_1 \leq \| f \|_1 \| g \|_1$ gives
\[
\int_0^\infty \| P_m f (t) \|_p \d t \lec 2^{-2m} \int_0^\infty \| P_m g (t) \|_p \d t
\]
for every $m$. Multiplying both sides by $2^{2m}$, summing in $m$ and applying the Tonneli theorem gives that $\| f \|_{L^1 B_{p,1}^{s+2}} \lec \| g \|_{L^1 B_{p,1}^{s}}$. This and the equation \eqref{par_eq} prove \eqref{max_reg}, as required.

As a simple corollary we note that an argument analogous to \eqref{mult_temp} shows that 
\eqnb\label{mult_ineqs}
\| P_m f \|_p \lec_p \| f \|_p, \qquad  \| P_m D^\gamma f \|_p \lec_p 2^{|\gamma |m} \| P_m f \|_p, \qquad \| P_m \Delta f \|_p \sim_p 2^{2m} \| P_m f \|_p
\eqne
for every $m\geq 1$, $p\in [1,\infty ]$ (by taking taking, respectively, $M(\xi ) = \phi_m (\xi )$, $M(\xi ) = \xi^\gamma \phi_{m\pm 1}(\xi )$, $M(\xi ) = |\xi |^2 \phi_{m\pm 1}(\xi )$ and $M(\xi ) = |\xi |^{-2} \phi_{m\pm 1}(\xi )$).
In particular
 \eqnb\label{der_in_out}
\|  D^\gamma f \|_{B^s_{p,q}} \lec \| f \|_{B^{s+|\gamma |}_{p,q}}, \qquad \|  \Delta f \|_{B^s_{p,q}} \sim \| f \|_{B^{s+2 }_{p,q}},
\eqne
for every $p,q\in [1,\infty ]$, $s \in \R$, given $\int f =0$.

We note that 
\eqnb\label{algebra_prop1}
\| fg \|_{B^s_{p,1}} \lec_\varepsilon \| f \|_{B^{d/p }_{p,1}} \| g \|_{B^{s}_{p,1}} 
\eqne
for $\varepsilon>0 $, $p\in [2, d)$, $s \in ( d/p , d/p] $, which can be proved in the same as the analogous claim for nonhomogeneous Besov spaces on $\R^3$, see Theorem~2(i) in Section~4.6.1 in \cite{rs_book}.

Finally we note that the $B^s_{p,1}$, for $s\in (0,1)$ norm is equivalent to the Lipschitz norm,   
\[
\| f \|_{B^s_{p,1}} \sim_{d,s,p} \| f \|_p + \int \left( \int \frac{|f(y)-f(x)|^p }{|y-x |^{p(d+s)}} \d y \right)^{\frac{1}{p}} \d x
\]
for $s\in (0,1)$, $p\in [1, \infty )$, see (18) on p.~169 in \cite{st_book}. See also (4) on p.~110 in \cite{triebel_83}.

 we can deduce from it that $B^s_{p,1}$ is invariant under diffeomorphisms for $s\in (0,1)$. Namely, given a diffeomorphism $Z \colon \T^d \to \T^d$ we have
\eqnb\label{diffeo_prop}
\| f \circ Z \|_{B^s_{p,1}} \sim_{s,d,p} C(\| \na Z \|_\infty , \| (\na Z )^{-1} \|_\infty ) \| f  \|_{B^s_{p,1}},
\eqne
for $s\in (0,1)$, $p\in [1,\infty ]$, by applying the change of variable $y\mapsto Z(y)$, using the Mean Value Theorem and estimating the Jacobian by the $L^\infty $ norms of $\na Z$ and $(\na Z)^{-1}$, see  Lemma~2.1.1 in \cite{dan_muc_memoir} for details. In what follows we will apply \eqref{diffeo_prop} for $s\coloneqq d/p-1 $, which belongs to $ (0,1)$, due to our restriction on $p$, namely $p\in (\min(d/2,2), d)$.

In what follows we will use a shorthand notation
\[
L^p B^s \equiv L^p ((0,\infty ); B^s_{p,1} (\T^d ) ).
\]

\section{Proof of Theorem~\ref{thm_main}}\label{sec_pf_thm} 

In this section we prove Proposition~\ref{prop_main}, which is equivalent to Theorem~\ref{thm_main} (see Section~\ref{sec_equiv} below).
We first consider the following compressible Stokes system,
\eqnb\label{lin_sys}
 \begin{array}{lcr }
  a_t + \div u =h & \mbox{ in } & \T^d \times \R_+,\\[7pt]
  u_t -\nu \Delta u + \nabla (K  a ) =g &
  \mbox{ in } & \T^d \times \R_+,\\[7pt] 
  a|_{t=0}=a_0, \qquad u|_{t=0} = u_0 & \mbox{ at } & \T^d,
 \end{array}
\eqne
where $g,h$ are given. This system is a linearization of \eqref{eqs_lagrangian1}, and the following lemma determines the types of spaces which we will use to estimate $u$ and $a$.

\begin{lemma}[Solution of the linear system]\label{lem_lin_system}
Given $s\in \R$, $p\in [1,\infty ]$, $a_0\in B^{s+1}_{p,1}$, $u_0\in B^{s}_{p,1}$, $g\in L^1 B^s_{p,1}$, $h\in L^1 B^{s+1}_{p,1}$ the system \eqref{lin_sys} admits a unique solution $(a,u)$ such that 
\eqnb\label{lin_ests}\begin{split}
\| a \|_{L^\infty B^{s+1}_{p,1}} + \| a_t \|_{L^1 B^{s+1}_{p,1} } &+\| a - \{ a \} \|_{L^1 B^{s-1}_{p,1} }+\|u\|_{L^\infty B^s_{p,1}} +\|u_t\|_{L^1B^s_{p,1}} + \|u-\{ u \} \|_{L^1B^{s+2}_{p,1}} \\
& \lec_{ \nu} \| a_0 \|_{B_{p,1}^{s+1}} + \| u_0 \|_{B_{p,1}^{s}} + \| h \|_{L^1 B^{s+1}_{p,1}}+ \| g \|_{L^1 B^{s}_{p,1}}.
\end{split}
\eqne
\end{lemma}
The lemma can be proved by first taking $\div$ of the second equation to obtain an evolution equation for $d\coloneqq \div u$. Taking $\p_t$ of the resulting PDE and substituting $a_t$ from the first equation we obtain an autonomous PDE on $d$, which we can solve by translating it into a family of second order ODEs for the Fourier coefficients of $d$. This allows us to find $a$ from the first equation. We can then use it to find $u$ from the second equation. 
\begin{proof}
We first note that we can assume that $g=0$ and $u_0=0$. Indeed, otherwise, we denote by $\widetilde{u}$ the solution of the heat equation with initial data $u_0$ and forcing $g$, i.e. we set
\begin{equation}\label{eq:u-tilde}
\widetilde{u}_k (t)  \coloneqq \int_0^t g_k \ee^{-\nu k^2 (t-s)} \d s + \ee^{-\nu k^2 t} u_{0k}. 
\end{equation}
Maximal regularity \eqref{max_reg} gives that
\eqnb\label{max_reg_tilde}
\|\widetilde{u}_t \|_{L^1 B^s_{p,1}} + \|  \Delta \widetilde{u} \|_{L^1 B^{s}_{p,1}} \lec \| g \|_{L^1 B^{s}_{p,1}}
+\|u_0\|_{B^s_{p,1}}.
\eqne
Then $(a,u-\widetilde{u})$ satisfies \eqref{lin_sys} with $g=0$, $(u-\widetilde{u})_{t=0} =0$ and the right-hand side of the equation for $a$ equal 
\eqnb\label{new_h}
\widetilde{h} \coloneqq h- \div \, \widetilde{u}.
\eqne
Note that 
\[
\| \widetilde{h} \|_{L^1 B^s_{p,1}} \lec \| {h} \|_{L^1 B^s_{p,1}} + \| \Delta \widetilde{u} \|_{L^1 B^{s-1}_{p,1}}  \lec \| {h} \|_{L^1 B^s_{p,1}} + \| g \|_{L^1 B^{s-1}_{p,1}}  + \| u_0 \|_{B^{s-1}_{p,1}}  
\]
for all $s$. Thus, if the lemma is valid in the homogeneous case $g=0$, $u_0=0$, then it is also valid in the inhomogeneous case. We can thus assume that $g=0$ and $u_0=0$.\\

Taking $\div $ of the second equation of \eqref{lin_sys} and setting $d\coloneqq \div \, u$ we obtain 
\eqnb\label{eq_for_a}
d_{t} - \nu \Delta d -(a-\{ a \} ) = 0 ,
\eqne
where we also used the fact that $\Delta (K a )=-(a-\{ a \} )$. Taking $\p_t$ and recalling that $a_t = h-d$ we obtain
\[
d_{tt} - \nu \Delta d_t +d  = -(h-\{ h \} )  
\]
with initial data $d|_{t=0}=0$ and from (\ref{eq_for_a})
$d_t|_{t=0}=a_0$. 

In terms of Fourier coefficients we obtain a $2$nd order ODE
\eqnb\label{eq_for_dk_fourier}
\p_{tt} d_k + \nu k^2 \p_t d_k + d_k = - h_k
\eqne
for $k\ne 0$. (Note that $d_0=\int \div u =0$.) The roots of the characteristic polynomial $\lambda^2 + \nu k^2 \lambda  + 1$ are 
\eqnb\label{lambda_roots}
\lambda^\pm_k = (-\nu k^2 \pm \sqrt{\nu^2 k^4 -4} )/2,
\eqne
where $k^2 \coloneqq k_1^2 + k_2^2 +\ldots + k_d^2$. Let us first assume that $\nu^2 k^4 \ne 4$ for all $k\in \Z^d\setminus \{ 0 \}$. Then there exists $C_\nu>0$ such that
\eqnb\label{lambdas_separated}
|\lambda^+_k-\lambda^-_k | \geq C_\nu  k^2 \qquad \text{ for } k\in \Z^3 \setminus \{ 0 \}.
\eqne
Note also that $\re\, \lambda_k^+, \re \,\lambda_k^- <0$ with 
\eqnb\label{real_parts_lambda_k_s}
|\lambda_k^+ | \sim |\re \, \lambda_k^+| \sim C_\nu k^{-2}, \qquad  |\lambda_k^- | \sim |\re\, \lambda_k^- |   \sim {C_\nu }k^2 
\eqne
for $k\ne 0$. We note that the above behavior of the roots $\lambda_k^+, \lambda_k^-$ determines the proprieties of the spectrum of the operator coming from the linear system (\ref{lin_sys}). In particular, as mentioned in the introduction, we emphasize that although the $d$-dimensional torus $\T^d$ is bounded, the spectrum is not separated from zero, which exclude possibility of the exponential decay of solutions.

We can now write the explicit form of the solution,
\begin{equation}\label{eq:d}
d_k (t) = A_k \ee^{\lambda^+_k t } + B_k \ee^{\lambda_k^- t} +\frac{1}{\lambda_k^+-\lambda_k^-}\int_0^t h_k (s) \left(\ee^{\lambda_k^+ (t-s)} - \ee^{\lambda_k^- (t-s)}  \right)  \d s,
\end{equation}
where $A_k, B_k \in \RR$ are such that
\[
\begin{pmatrix}
1&1 \\
\lambda_k^+ &\lambda_k^-
\end{pmatrix}
\begin{pmatrix}
A_k \\ B_k
\end{pmatrix}=\begin{pmatrix}
0\\
a_{0k}
\end{pmatrix}
\]
In particular, \eqref{lambdas_separated} gives that
\eqnb\label{ak_bk_bounds}
|A_k | , |B_k | \leq  C_\nu k^{-2} |a_{0k}|
\eqne
If $\nu k^2 =2 $ for some $k\in \Z^d\setminus \{ 0\}$ then
\begin{equation*}
d_k (t) = A_k \ee^{-t } + B_k t \ee^{- t} +\int_0^t \int_0^s  h_k (\tau )  \ee^{-(t-\tau ) } \d \tau\,  \d s
\end{equation*}
for such $k$, where
$A_k \coloneqq 0,
B_k \coloneqq  a_{0k}$.
In particular \eqref{ak_bk_bounds} follows in this case as well.

Thus considering the modes $k\sim 2^m$ we can use \eqref{each_mode_exp} to obtain  
\eqnb\label{d_est1}\begin{split}
\int_0^\infty \| P_m d (t) \|_p \d t &\lec_\nu 2^{-2m} \| P_m a_0 \|_p  \int_0^\infty  \left( \ee^{-c2^{-2m} t} + \ee^{-c2^{2m}  t} \right)\d t \\
&+ 2^{-2m} \left\| \| P_m  h \|_p \ast \ee^{-c2^{-2m} t } +  \| P_m  h \|_p \ast \ee^{-c2^{2m} t } \right\|_{L^1_t}\\
&\lec  \| P_m a_0 \|_p   + \int_0^\infty  \| P_m  h (t)  \|_p \d t 
\end{split}
\eqne
for $m\geq 1$, which $\| P_0 d(t) \|_p \sim d_0 (t) =0$ for all $t>0$. 

Similarly, 
\[
d_k' (t) = A_k\lambda_k^+ \ee^{\lambda^+_k t } + B_k\lambda_k^- \ee^{\lambda_k^- t} +\frac{1}{\lambda_k^+-\lambda_k^-}\int_0^t h(s) \left(\lambda_k^+ \ee^{\lambda_k^+ (t-s)} - \lambda_k^- \ee^{\lambda_k^- (t-s)}  \right)  \d s
\]
for every $k\in \Z^3 \setminus \{ 0 \}$, and, analogously to \eqref{d_est1}, \eqref{each_mode_exp} gives that ($k\sim 2^m$)
\[\begin{split}
\int_0^\infty \| P_m d_t (t) \|_p \d t &\lec 2^{-2m} \| P_m a_0 \|_p +  2^{-2m}  \int_0^\infty  \| P_m  h (t)  \|_p \d t 
\end{split}
\]
for $m\geq 1$. This and \eqref{d_est1} implies that
\eqnb\label{d_ests}
\begin{split}
\| d \|_{L^1 B_{p,1}^s } &\lec_{\nu} \| a_0 \|_{B_{p,1}^{s}}  + \| h \|_{L^1 B^s_{p,1}}\\
\| d \|_{L^\infty B_{p,1}^s }+\| d_t \|_{L^1 B_{p,1}^s } &\lec_{\nu} \| a_0 \|_{B_{p,1}^{s-2}}+ \| h \|_{L^1 B^{s-2}_{p,1}}
\end{split}
\eqne
Moreover, using \eqref{eq_for_a} we see that $a-\{ a \} =d_t - \Delta d$, which implies that 
\eqnb\label{a_est1}
\| a -\{ a \} \|_{L^1 B^s_{p,1} } \lec_{\nu }\| d_t \|_{L^1 B^s_{p,1} } +\| d \|_{L^1 B^{s+2}_{p,1} } \lec_{ \nu} \| a_0 \|_{B_{p,1}^{s+2}}  + \| h \|_{L^1 B^{s+2}_{p,1}},
\eqne
where we applied \eqref{embeddings} to write $\| d_t \|_{L^1 B^s_{p,1} } \lec  \| d_t \|_{L^1 B^{s+4}_{p,1} } $, and used \eqref{d_ests}. On the other hand, $a_t = -d + h$, which gives that
\eqnb\label{a_est2}
\| a \|_{L^\infty B^s_{p,1}} + \| a_t \|_{L^1 B^s_{p,1} } \lec \| a_0 \|_{B_{p,1}^{s}} + \| d \|_{L^1 B^s_{p,1} } +\| h \|_{L^1 B^{s}_{p,1} } \lec \| a_0 \|_{B_{p,1}^{s}}  + \| h \|_{L^1 B^s_{p,1}}.
\eqne
Moreover, recalling that $f\mapsto \nabla K  f $ is an operator of order $-1$, we can use maximal regularity \eqref{max_reg} of the second equation of \eqref{lin_sys},
\begin{equation}\label{eq:uu}
u_t - \Delta u = -\nabla K a 
\end{equation}
to obtain
\[
\|u \|_{L^\infty B^s_{p,1}} + \|u_t\|_{L^1 B^s_{p,1}} + \|\Delta u\|_{L^1 B^{s}_{p,1}}  \lec_\nu \| a \|_{L^1 B^{s-1}_{p,1}} \lec_\nu \| a_0 \|_{B_{p,1}^{s+1}} + \| h \|_{L^1 B^{s+1}_{p,1}},
\]
where we used \eqref{a_est1} in the last inequality. This, \eqref{a_est1} and \eqref{a_est2} give \eqref{lin_ests}, as required. 

The estimates \eqref{lin_ests} prove uniqueness of solutions. As for existence, we first define $d$ by \eqref{eq:d} and then we set $a\coloneqq \{ a_0 \} +\int_0^t \{ h(s) \} \d s + d_t-\Delta d$. 
\end{proof}

We can now prove Proposition~\ref{prop_main} 

\begin{proof}[Proof of Proposition~\ref{prop_main}.]
We rewrite \eqref{eqs_lagrangian1} in the form
\eqnb\label{eq_lagr}
\begin{split}
a_t + \div \, u &= -a\div u+(1+a)(\div -\div_u ) u \\
&=: h\\
u_t -\Delta u + \na (-\Delta )^{-1} a &= a \na (-\Delta )^{-1} a-au_t +  (\Delta_u - \Delta )u \\
& \qquad +(1+a)\left( \na (-\Delta )^{-1} - \na_u (-\Delta_u )^{-1} \right) a\\
&=: g,
\end{split}
\eqne
and note that Lemma~\ref{lem_lin_system} gives that 
\begin{multline}\label{cons_of_lem}
\| a \|_{L^\infty \BB } + \| a_t \|_{L^1 \BB } + \| a-\{ a \} \|_{L^1 \UB } + \| u \|_{L^\infty \B}  +  \| u_t \|_{L^1 \B } +  \| u-\{ u \} \|_{L^1 \BBB }
\\
\lec \| a_0 \|_{\BB }+ \| u_0 \|_{\B }+  \| h \|_{L^1 \BB } + \| g \|_{L^1 \B }.
\end{multline}
Assuming that 
\eqnb\label{smallness_u_L1B2}
\| u - \{ u \} \|_{L^1 \BBB } \leq \gamma,
\eqne
where $\gamma \in (0,1)$ is a sufficiently small constant, we show in Step 1 below that
\eqnb\label{for_step1}
\| h \|_{L^1 \BB } \lec \| a \|_{L^\infty \BB } \| u - \{ u \} \|_{L^1 \BB }+\left( 1+ \| a \|_{L^\infty  \BB  } \right) \| u -\{ u \} \|_{L^1 \BBB  }^2
\eqne
and in Step 2 that
\begin{multline} \label{for_step2}
\| g \|_{L^1 \B } \lec \| a \|_{L^\infty \BB  } \left( \|a -\{ a \} \|_{L^1 \UB }  + \| u_t \|_{L^1 \B  } \right) + \| u - \{ u\} \|_{L^1 \BBB }^2
\\
+ \left( 1+ \| a \|_{L^\infty \BB  }\right) \| u-\{ u \}  \|_{L^1 \BBB  } \| a- \{ a \} \|_{L^1 \UB }
\end{multline}
Thanks to these estimates we can use \eqref{cons_of_lem} to obtain the a~priori bound 
\eqnb\label{apriori}\begin{split}
\| a& \|_{L^\infty \BB }  + \| a_t \|_{L^1 \BB } + \| a - \{ a \}  \|_{L^1 \UB } + \| u \|_{L^\infty \B }  +  \| u_t \|_{L^1 \B }+  \| u -\{ u \} \|_{L^1 \BBB } \\
&\lec \| a_0 \|_{\BB }+ \| u_0 \|_{\B }+  \| h \|_{L^1 \BB } + \| g \|_{L^1 \B }\\
&\lec \| a_0 \|_{\BB }+ \| u_0 \|_{\B } \\
&\hspace{1cm}+ \| a \|_{L^\infty \BB  } \left( \|a-\{ a \}  \|_{L^1 \UB }  + \| u_t \|_{L^1 \B  } + \| u-\{ u \} \|_{L^1 \BBB  } \right) \\
&\hspace{1cm}+\left( 1+ \| a \|_{L^\infty \BB  }\right) \| u -\{ u \} \|_{L^1 \BBB  } \left( \| a- \{ a \} \|_{L^1 \UB } + \| u - \{ u \} \|_{L^1 \BBB  } \right).
\end{split}
\eqne

Note that the right-hand sides of \eqref{for_step1}, \eqref{for_step2} are at least quadratic in $(a,u)$. This allows us to use the a~priori bound to prove claim using Banach Contraction Theorem, which we discuss in Step 3 for the sake of completeness.\\

\noindent\texttt{Step 1.} We prove \eqref{for_step1}.\\

We note that
\eqnb\label{step1a}
\| a \div \, u \|_{L^1 \BB } \lec \| a \|_{L^\infty \BB } \| u -\{ u \} \|_{L^1 \BBB }.
\eqne
as for the other ingredient of $h$, we first use \eqref{def_A} to expand $A$ as the Neumann series
\eqnb\label{neumann}
A-I=\sum_{k\geq 1} \left(-\int_0^t \na u \right)^k,
\eqne
where $I$ denotes the $d\times d$ identity matrix. Taking the $L^\infty \BB $ norm we see that
\eqnb\label{smallness_I_A}
\| A- I \|_{L^\infty \BB  } \leq \sum_{k\geq 1}\left\| \int_0^t \na u \right\|_{L^\infty \BB}^k \leq \sum_{k\geq 1}\left\|  u -\{ u \} \right\|_{L^1 \BBB }^k \leq 2 \| u - \{ u \} \|_{L^1 \BBB } \leq 2\gamma, 
\eqne
provided $\gamma <1/2$, where we used \eqref{smallness_u_L1B2}.
We note that, since $\na = \left( \frac{\d X }{\d y}\right)^T \na_u$ we have $\na_u =A^T \na $. Thus
\eqnb\label{div_euler}
(\div - \div_u ) u =  (\delta_{ij} - A_{ji} ) \p_j u_i ,
\eqne
and consequently
\eqnb\label{step1b}\begin{split}
\| (1+a) (\div- \div_u )u \|_{L^1 \BB  } &\lec \left( 1+ \| a \|_{L^\infty \BB  } \right) \| (\delta_{ij}- A_{ji} ) \p_j u_i \|_{L^1 \BB  } \\
&\lec \left( 1+ \| a \|_{L^\infty \BB  } \right) \| I-A \|_{L^\infty \BB  } \| u -\{ u \} \|_{L^1 \BBB  }  \\
&\lec  \left( 1+ \| a \|_{L^\infty \BB  } \right) \| u - \{ u \}  \|_{L^1 \BBB  }^2,
\end{split}
\eqne
as required, where we used \eqref{smallness_I_A} in the last step.\\

\noindent\texttt{Step 2.} We prove \eqref{for_step2}.\\

As for the first two ingredients of $g$ we obtain
\eqnb\label{g_first_2}
\| a \na (-\Delta )^{-1} a -au_t\|_{L^1 \B } \lec \| a \|_{L^\infty \BB  } \left( \|a -\{ a \} \|_{L^1 \UB }  + \| u_t \|_{L^1 \B } \right),
\eqne
as required.

As for the remaining two ingredients we first note that
\eqnb\label{Delta_comm_pre} \begin{split}
\Delta_u -\Delta &= \div_u \na_u  - \div \na = \div (\na_u -\na ) + (\div_u - \div )\na_u  \\
&=\p_j ((A_{ij} -\delta_{ij} ) \p_i)+ (A_{ji}-\delta_{ij})\p_j (\p_{i}+ (\delta_{ik}-A_{ki})\p_k )
\end{split}\eqne
which gives that 
\eqnb\label{Delta_commutator}\begin{split}
\| (\Delta_u - \Delta )u \|_{B^s} &\leq \| (A^T-I) \na u \|_{B^{s+1}} + \| A-I \|_{\BB  } \| \na u + (I-A^T )\na u \|_{B^{s+1}}\\
&\lec \| A-I \|_{\BB } \left( 1+ \| A-I \|_{\BB } \right) \| u -\{ u \} \|_{B^{s+2}}
\end{split}
\eqne
for each time and $s=d/p-2,d/p-1 $, where we used \eqref{algebra_prop1} in the first line. We note in passing that \eqref{Delta_commutator} is the main reason for our restriction on the range of $p$, due to the order restriction in the product law \eqref{algebra_prop1}.
Taking $s=d/p-1$ we can estimate the third ingredient of $g$,
\eqnb\label{Delta_comm_conseq}
\begin{split}
\| (\Delta_u - \Delta )u \|_{L^1 \B } &\leq  \| A-I \|_{L^\infty \BB } \left( 1+ \| A-I \|_{L^\infty \BB } \right) \| u- \{ u \}  \|_{L^1 \BBB } \\
&\lec \| u - \{ u \}  \|_{L^1 \BBB }^2,
\end{split}
\eqne
where we used \eqref{smallness_I_A} and \eqref{smallness_u_L1B2}.

On the other hand, taking $s=d/p-2$ in \eqref{Delta_commutator} gives an elliptic estimate
\eqnb\label{ellip_est}
\| (-\Delta_u )^{-1} w \|_{\BB } \lec \| w \|_{\UB }
\eqne
for $w\in \UB $ with $\int w =0$. Indeed letting $f\coloneqq (-\Delta_u )^{-1} w$ we see that $f$ satisfies the Poisson equation
\[
-\Delta f = -\Delta_u  f + (\Delta_u - \Delta ) f = w +(\Delta_u - \Delta ) f 
\]
on the torus, which, after noting that $w=(-\Delta ) (-\Delta )^{-1} w$, gives
\[
-\Delta ((-\Delta_u )^{-1}-(-\Delta )^{-1})w = (\Delta_u -\Delta ) f.
\]
Taking the $\UB $ norm gives
\eqnb\label{diff_of_inv_Lapl}
\| \left( (-\Delta_u )^{-1} - (-\Delta )^{-1} \right)  w \|_{\BB } \lec \| (\Delta_u -\Delta ) f \|_{\UB } \lec \| A-I \|_{\BB } \left( 1+ \| A-I \|_{\BB } \right) \| f \|_{\BB },
\eqne
where we used \eqref{Delta_commutator} with $s=d/p-2$ in the second inequality. In particular 
\[\begin{split}
\| f \|_{\BB } &\lec \| w \|_{\UB } + \| (-\Delta_u )^{-1}w - (-\Delta )^{-1} w \|_{\BB }  \lec  \| w\|_{\UB } + \gamma  \|f \|_{\BB }
\end{split}
\]
where we used \eqref{smallness_I_A} in the last inequality. The elliptic estimate \eqref{ellip_est} follows if $\gamma $ is chosen sufficiently large so that the last term can be absorbed by the left-hand side. 

The last ingredient of $g$ can now be estimated by noting the identity
\[
\na (-\Delta )^{-1} - \na_u (-\Delta_u )^{-1} = \na ((- \Delta )^{-1}-(-\Delta_u )^{-1} )  -  (\na_u -\na ) (-\Delta_u )^{-1} ,
\]
which gives that
\eqnb\label{cubic_term}
\begin{split}
\| &(1+a) \left( \na (-\Delta )^{-1} - \na_u (-\Delta_u )^{-1} \right) a \|_{L^1 \B } \\
&\lec \left( 1+ \| a \|_{L^\infty \BB  }\right) \left( \| \left( (-\Delta )^{-1} - (-\Delta_u )^{-1}\right) ( a-\{ a \} )\|_{L^1 \BB } \right. \\
&\left.\hspace{6cm}+ \| (\na_u -\na)(-\Delta_u )^{-1} (a-\{ a \} ) \|_{L^1 \B  } \right) \\
&\lec \left( 1+ \| a \|_{L^\infty \BB }\right) \| I-A \|_{L^\infty \BB  } \left( 1+\| A-I \|_{L^\infty \BB } \right) \| (-\Delta_u )^{-1} (a- \{ a \} ) \|_{L^1 \BB }  \\
&\lec  \left( 1+ \| a \|_{L^\infty \BB  }\right) \| u -\{ u \} \|_{L^1 \BBB  } \| a- \{ a \} \|_{L^1 \UB },  
\end{split}
\eqne
as required, where we used \eqref{diff_of_inv_Lapl}, the product rule \eqref{algebra_prop1} and the fact that $\na_u - \na = (A^T-I)\na $ in the second inequality, as well as \eqref{smallness_I_A} in the last line.\\

\noindent\texttt{Step 3.} We prove the claim.\\

We set 
\[
\| (u,a) \| \coloneqq   \| a \|_{L^\infty \BB } + \| a_t \|_{L^1 \BB  }  + \| a - \{ a \} \|_{L^1 \UB  }
\]
\[+ \| u \|_{L^\infty \B } + \| u_t \|_{L^1 \B } + \| u -\{ u \} \|_{L^1 \BBB }  
\]
and
\[
V\coloneqq \{ (u,a ) \colon \| (u,a) \| < \infty \}.  
\]
Then $V$ equipped with the norm $\| \cdot \|$ is a Banach space. Given $(\vv,\aa ) \in V$, we let $S(\vv , \aa )$ denote the solution of the linear system \eqref{lin_sys} with 
\[
\begin{split}
h &\coloneqq -\aa \div \vv +(1+\aa )(\div -\div_{\vv } ) \vv  \\
g &\coloneqq  \aa  \na (-\Delta )^{-1} \aa - \aa \vv_t +  (\Delta_{\vv } - \Delta )\vv  +(1+\aa )\left( \na (-\Delta )^{-1} - \na_{\vv } (-\Delta_{\vv } )^{-1} \right) \aa
\end{split}
\]
By \eqref{apriori} we see that $S\colon V\to V$. By Lemma~\ref{lem_lin_system} there exists $C>0$ such that $\| S(0,0)  \| \leq {C} (\| a_0 \|_{\BB } + \| u_0 \|_{\B } )/2 \leq {C\epsilon }/2$. By the a~priori estimate \eqref{apriori} we obtain that $\| S(\vv , \aa ) - S(0,0) \| \lec \| (\vv , \aa ) \|^2 \leq C^2 \epsilon^2$ for every $(\vv , \aa ) \in B\coloneqq B(0,C\epsilon )$. Thus $S$ maps $B$ into itself for sufficiently small $\epsilon >0$, as $\| S(\vv , \aa ) \| \leq \frac{C\epsilon }2 + \| S(\vv , \aa ) - S(0,0) \| \leq C \epsilon$ for $(\vv, \aa)\in B$. We show below that $S$ is a contraction on $B$, namely that
\eqnb\label{contr}
\| S(\vv , \aa ) - S(\ww , \bb ) \| \leq \frac12 \| (\vv - \ww ,\aa-   \bb ) \| =: \frac12 d
\eqne
for all $(\vv , \aa ), (\ww , \bb )\in B$, given $\epsilon >0$ is chosen sufficiently small. Banach Contraction Theorem then gives the claimed existence and uniqueness result.\\

Letting $(u,a) \coloneqq S(\vv , \aa ) - S(\ww , \bb )$ we see that $(u,a)$ is a solution to the problem
\eqnb\label{lin_sys_ua}
 \begin{array}{lcr }
  a_t + \div u = \delta h & \mbox{ in } & \T^3 \times \R_+,\\[7pt]
  u_t -\nu \Delta u + \nabla (K  a ) =\delta g &
  \mbox{ in } & \T^3 \times \R_+
 \end{array}
\eqne
with homogeneous initial conditions $a(0)=0$, $u(0)=0$, where
\[\begin{split}
    \delta h&\coloneqq -\aa \div \vv +(1+\aa )(\div -\div_{\vv } ) \vv  -\bb \div \ww +(1+\bb )(\div -\div_{\ww } ) \ww \\
    &=- (\aa - \bb )  \div \vv  - \ww \div (\vv - \ww )
    +(\aa - \bb ) ( \div -\div_{\vv } ) \vv \\
    &\hspace{5cm}+    (1+\bb ) \left( (\div -\div_{\vv } ) \vv -(\div -\div_{\ww } ) \ww\right)\\
    &= - (\aa - \bb )  \div \vv  - \bb \div (\vv - \ww ) +(\aa - \bb )  (\div -\div_{\vv} ) \vv\\
    &\hspace{5cm} +   (1+\bb ) \left((\div -\div_{\ww } ) (\vv-\ww )\right)+(1+\bb ) \left( (\div_{\ww } -\div_{\vv } ) \vv \right) 
\end{split}
\]
and
\[
 \begin{split}
    \delta g&\coloneqq 
      \aa  \na (-\Delta )^{-1} \aa -\aa  \vv_t +  (\Delta_{\vv } - \Delta )\vv  
           +(1+\aa )\left( \na (-\Delta )^{-1} - \na_{\vv } (-\Delta_{\vv } )^{-1} \right) \aa \\
     &\hspace{0.5cm}-\big( \bb  \na (-\Delta )^{-1} \bb -\bb \ww_t +  (\Delta_{\ww } - \Delta ) \ww 
       +(1+\bb )\left( \na (-\Delta )^{-1} - \na_{\ww } (-\Delta_{\ww } )^{-1} \right) \bb  \big)\\
     &= \underbrace{ (\aa - \bb ) \na (-\Delta )^{-1} \aa + \bb \na (-\Delta )^{-1} (\aa - \bb )}_{=: \delta g_1} 
\underbrace{       - (\aa - \bb ) \vv_t  -\bb (\vv-\ww )_t }_{=: \delta g_2}
       \\
      &\hspace{0.5cm}+\underbrace{(\Delta_\vv - \Delta )(\vv - \ww )}_{=: \delta g_3} +\underbrace{ (\Delta_\vv -\Delta_\ww )\ww}_{=: \delta g_4} \\\
      &\hspace{0.5cm} + \underbrace{(\aa - \bb ) \left( \na (-\Delta )^{-1} - \na_{\vv } (-\Delta_{\vv } )^{-1} \right) \aa }_{=: \delta g_5} +\underbrace{(1+\bb )\left(  \na (-\Delta )^{-1} - \na_{\vv } (-\Delta_{\vv } )^{-1} \right) (\aa - \bb )}_{=: \delta g_6} \\
      &\hspace{0.5cm} + \underbrace{ (1+\bb ) \left( \na_\ww (-\Delta_\ww )^{-1} - \na_{\vv } (-\Delta_{\vv } )^{-1} \right) \bb }_{=: \delta g_7}
\end{split}
\]
In the remainder of the proof we verify that
\[
\| \delta h \|_{L^1 \BB } + \| \delta g \|_{L^1 \B } \lec \epsilon d 
\]
whenever $(\vv , \aa ) , (\ww, \bb) \in B$.  This and Lemma~\ref{lem_lin_system} proves the required contraction property \eqref{contr} if $\epsilon >0$ is chosen sufficiently small.\\

Looking at the structure of $\delta h$ we see that the first four terms can be bounded in $\| \cdot \|_{L^1 \BB }$ in the same way as in Step 1 above (recall \eqref{step1a} and \eqref{step1b}), to give the upper bound
\[\begin{split}
\| \aa - \bb \|_{L^\infty \BB  }&\| \vv - \{ \vv \} \|_{L^1 \BBB  }+\|  \bb \|_{L^\infty \BB  }\| \vv -\ww - \{ \vv - \ww \} \|_{L^1 \BBB  }  \\
&+\|\aa - \bb \|_{L^\infty \BB  } \| I-\overline{A} \|_{L^\infty \BB }  \| \vv - \{ \vv \} \|_{L^1 \BBB  } \\
&+ \left( 1+ \| \bb \|_{L^\infty \BB  }   \right) \| I-\overline{\overline{A}} \|_{L^\infty \BB }  \| \vv - \ww -\{ \vv - \ww \}\|_{L^1 \BBB  }\\
&\hspace{3cm} \lec \epsilon \left( \| \aa - \bb \|_{L^\infty \BB  } + \| \vv - \ww -\{ \vv - \ww \}\|_{L^1 \BBB  } \right)  \lec \epsilon d,
\end{split}
\]
where
\[
{\overline{A}} \coloneqq I + \sum_{k\geq 1} \left( - \int_0^t \na \vv \right)^k\qquad \text{ and }\qquad \overline{\overline{A}} \coloneqq I + \sum_{k\geq 1} \left( - \int_0^t \na \ww \right)^k,
\]
recall the Neumann expansion \eqref{neumann}.

As for the last ingredient of $\delta h$ we recall the Neumann series \eqref{neumann} and the algebraic identity $a^k - b^k = (a-b) \sum_{m=0}^{k-1} a^{k-m}b^m$ to write 
\[\begin{split}
\overline{\overline{A}} - \overline{A} & = \sum_{k\geq 1} \left( \left(-\int_0^t \na \ww  \right)^k -\left(-\int_0^t \na \vv  \right)^k  \right)\\
&= \left(-\int_0^t \na (\ww - \vv )  \right) \sum_{k\geq 1} \sum_{m=0}^{k-1}  \left(-\int_0^t \na \ww  \right)^{k-m} \left(-\int_0^t \na \vv  \right)^m,
\end{split}\]
 Thus, taking the $L^\infty \BB $ norm gives
\eqnb\label{A-A} \begin{split}
\| \overline{\overline{A}} - \overline{A} \|_{L^\infty \BB  } &\leq \| \vv-\ww -\{ \vv - \ww \}\|_{L^1 \BBB } \sum_{k\geq 1} \sum_{m=0}^{k-1} \| \ww - \{ w\} \|_{L^1 \BBB  }^{k-1-m} \| \vv - \{ \vv \} \|_{L^1 \BBB  }^m 
\\
& \lec \| \vv-\ww -\{ \vv - \ww \} \|_{L^1 \BBB  } \sum_{k\geq 1 } \sum_{m=0}^{k-1} \epsilon^{k-1-m} \epsilon^m  \\
&\lec \| \vv-\ww -\{ \vv - \ww \} \|_{L^1 \BBB  } \sum_{k\geq 1 } k \epsilon^{k-1} \lec  \| \vv-\ww -\{ \vv - \ww \} \|_{L^1 \BBB  }.
\end{split}
\eqne
Hence, recalling \eqref{div_euler}, we obtain
\eqnb\label{lms_deltah}
\begin{split}
    \| (1+\bb ) \left( \div_{\ww } -\div_{\vv } ) \vv \right)\|_{L^1 \BB } &\leq  \left( 1+ \| \bb \|_{L^\infty \BB  } \right)  \| A_\ww - A_\vv \|_{L^\infty \BB } \| \vv -\{ \vv \} \|_{L^1 \BBB  }  \\
    &\lec  \| \vv-\ww -\{ \vv - \ww \} \|_{L^1 \BBB  } \| \vv -\{ \vv \} \|_{L^1 \BBB }  \lec \epsilon d,
    \end{split}
    \eqne
as required.\\

As for $\delta g$ we have
\[\begin{split}
\| \delta g_1 + \delta g_2  \|_{L^1 \B  } & \lec \| (\aa - \bb ) \na (-\Delta )^{-1} \aa\|_{L^1 \B  } + \| \bb \na (-\Delta )^{-1} (\aa - \bb ) \|_{L^1 \B  } \\
&\hspace{5cm}  + \| (\aa - \bb ) \vv_t \|_{L^1 \B  } + \| \bb (\vv-\ww )_t \| _{L^1 
\B  }\\
& \lec \| \aa - \bb \|_{L^\infty \BB  } \left( \| \aa -\{ \aa \} \|_{L^1 \UB } + \| \bb - \{ \bb \} \|_{L^1 \UB } + \| v_t  \|_{L^1 \B  }   \right)
\\
&
\hspace{5cm} + \| \bb \|_{L^\infty \BB   } \| \vv_t - \ww_t \|_{L^1 \B  }\\
&\lec \epsilon \| (\vv, \aa ) - (\ww , \bb ) \|
\end{split}
\]
as in \eqref{g_first_2}. On the other hand, the first inequality in \eqref{Delta_comm_conseq} gives
\[\begin{split}
\| \delta g_3 \|_{L^1 \B } &\lec \| \overline{A} - I \|_{L^\infty \BB  } \left( 1+  \| \overline{A} - I \|_{L^\infty \BB  }\right) \| \vv-\ww -\{ \vv - \ww \}  \|_{L^1 \BBB }\\
& \lec \epsilon \| \vv-\ww -\{ \vv - \ww \} \|_{L^1 \BBB }.
\end{split}
\]
As for $\delta g_4$ we have, as in \eqref{Delta_comm_pre}
\eqnb\label{Delta_v-w}\begin{split}
\Delta_\vv - \Delta_\ww &= \div_\vv (\na_\vv - \na_\ww ) + (\div_\vv - \div_\ww ) \na \ww \\
&= \overline{A}_{ji} \p_j \left( \left(\overline{A}_{ki} - \overline{\overline{A}}_{ki} \right) \p_k \right)+ \left(\overline{A}_{ji} - \overline{\overline{A}}_{ji} \right)\p_j \p_i,
\end{split}\eqne
which implies (as in \eqref{Delta_commutator}) that 
\eqnb\label{Delta_v-w_inBs}\begin{split}
\| (\Delta_\vv - \Delta_\ww )\ww \|_{B^s} &\leq \| \overline{A} \|_{\BB } \left\| \left( \overline{A} - \overline{\overline{A}} \right) \na \ww \right\|_{B^{s+1}} +  \left\| \left(\overline{A} - \overline{\overline{A}} \right):  D^2 \ww \right\|_{B^s}\\
&\lec \left( 1+ \| \overline{A} - I \|_{\BB } \right) \| \overline{A} - \overline{\overline{A}} \|_{\BB } \| \ww -\{ \ww \}  \|_{B^{s+2}}
\end{split}
\eqne
for $s=d/p-2, d/p-1$ and each fixed time, which in turn (similarly to \eqref{Delta_comm_conseq}) gives that
\[\begin{split}
\| \delta g_4  \|_{L^1 \B } &= \| (\Delta_\vv - \Delta_\ww) \ww  \|_{L^1 \B } \lec \left\| \overline{A} - \overline{\overline{A}} \right\|_{L^\infty  \BB  } \| \ww - \{ \ww \} \|_{L^1 \BBB  } \\
&\lec \epsilon \| \vv-\ww -\{ \vv - \ww \}  \|_{L^1 \BBB },
\end{split}
\]
where we used \eqref{A-A} in the last inequality.

As for $\delta g_5$ and $\delta g_6$ \eqref{cubic_term} gives that
\[\begin{split}
\| \delta g_5 + \delta g_6 \|_{L^1 \B } &\lec \| \aa - \bb \|_{L^\infty \BB  } \| \vv - \{ \vv \}  \|_{L^1 \BBB  } \| \aa - \{ \aa \} \|_{L^1 \UB }\\
& \qquad + \left( 1+ \| \bb \|_{L^\infty \BB  }\right) \| \vv - \{ \vv \} \|_{L^1 \BBB  } \| \aa - \bb - \{ \aa - \bb \} \|_{L^1 \UB }\\
&\lec \| \vv - \{ \vv \} \|_{L^1 \BBB  } \left(  \| \aa - \bb \|_{L^\infty \BB } + \| \aa - \bb - \{ \aa - \bb \}\|_{L^1 \UB } \right) .
\end{split}
\]
Finally, $\delta g_7 =(1+\bb ) \left( \na_\ww (-\Delta_\ww )^{-1} - \na_{\vv } (-\Delta_{\vv } )^{-1} \right) \bb$ is the most challenging term, where, as in \eqref{lms_deltah} above, we need to extract $v-w$ from the difference of differential operators $\na_\ww (-\Delta_\ww )^{-1} - \na_{\vv } (-\Delta_{\vv } )^{-1} $. To this end we need to explore the steps leading to \eqref{cubic_term} a bit further. Namely, setting
\[
f\coloneqq (-\Delta_\ww )^{-1} (\bb - \{\bb \} ) , \qquad g \coloneqq (-\Delta_\vv )^{-1} (\bb - \{ \bb \} ),
\]
we see that
\[
\delta g_7 = (1+\bb ) \left(  \na_\ww (f-g ) + (\na_\ww - \na_\vv ) g  \right), 
\]
 recalling \eqref{def_A} that $\na_\ww = \overline{\overline{A}}^T \na$ we obtain
\eqnb\label{dg7}\begin{split}
\| \delta g_7  \|_{ \B } &\lec \underbrace{ \left( 1 + \| \bb \|_{ \BB } \right) }_{\lec 1} \| \na_\ww (f-g ) + (\na_\ww - \na_\vv ) g  \|_{ \B }\\
& \lec  \underbrace{ \| \overline{\overline{A}} \|_{ \BB  } }_{\lec 1} \| f-g \|_{\BB } + \| \overline{\overline{A}} - \overline{A} \|_{ \BB  }  \|  g  \|_{ \BB }.
\end{split}
\eqne
In order to estimate the first term on the right-hand side we note that
\[
-\Delta (f-g) = (\Delta_\ww - \Delta )f - (\Delta_\vv -\Delta ) g = (\Delta_w - \Delta_v )f + (\Delta_v - \Delta ) (f-g) .
\]
This lets us use \eqref{Delta_v-w} with $s=-1$ (and $(\Delta_\vv - \Delta_\ww ) \ww$ replaced by $(\Delta_\ww - \Delta_\vv ) f$) and \eqref{Delta_commutator} with $s=-1$ (and $(\Delta_u - \Delta )u$ replaced by $(\Delta_v - \Delta ) (f-g)$) to obtain
\[
\begin{split}
\| f-g \|_{ \BB } &\lec \| (\Delta_\ww - \Delta_\vv )f \|_{ \UB } + \| (\Delta_v - \Delta ) (f-g )  \|_{\UB }   \\
&\lec \underbrace{ \left( 1+ \| \overline{\overline{A}} - I \|_{ \BB } \right)}_{\lec 1}  \|  \overline{\overline{A}}-\overline{A}  \|_{ \BB }  \| f \|_{ \BB }   +\underbrace{ \| \overline{A}-I \|_{\BB  }}_{\lec \epsilon } \underbrace{ \left( 1+ \| \overline{A}-I \|_{ \BB } \right)}_{\lec 1}  \| f-g \|_{ \BB } .
\end{split}
\]
Thus, for sufficiently small $\epsilon >0$ we can absorb the last term on the left-hand side to obtain that
\[
\| f-g \|_{ \BB } \lec \|  \overline{\overline{A}}-\overline{A}  \|_{\BB }  \| f \|_{ \BB }
\]
at each time. Applying this in \eqref{dg7} and integrating in time we obtain
\[
\| \delta g_7 \|_{L^1 \B } \lec \|  \overline{\overline{A}}-\overline{A}  \|_{ L^\infty \BB } \left(  \| f \|_{ L^1 \BB } +  \| g \|_{ L^1  \BB  }  \right) 
\]
\[
\lec \| \vv-\ww - \{ \vv - \ww \} \|_{L^1 \BBB  } \| \bb - \{ \bb \} \|_{L^1 \UB } \lec \epsilon \|  \vv-\ww - \{ \vv - \ww \} \|_{L^1 \BBB  } ,
\]
as required, where we used \eqref{A-A} and \eqref{ellip_est} in the second inequality.
 \end{proof}

\section{Equivalence of the Eulerian and Lagrangian formulations}\label{sec_equiv}
In this section we show the equivalence of Theorem~\ref{thm_main} and Proposition~\ref{prop_main}.\\

We first show that Theorem~\ref{thm_main} $\Rightarrow $ Proposition~\ref{prop_main}.\\

To this end, given a solution $(\rho , v)$ in the Eulerian coordinates, we need to construct a Lagrangian map $X=X(t,y)$, so that $a(t,y) \coloneqq \rho (t,X(t,y))-1$, $u(t,y) \coloneqq v (t,X(t,y))$ is a solution in the Lagrangian coordinates (that is a solution of \eqref{eqs_lagrangian1}). 

In order to construct the Lagrangian map, we first prove the following a~priori estimate: given $X(t)$ exists for all $t\geq 0$, and $\| \na X - I \|_{L^\infty \BB } $ is sufficiently small (see \eqref{assump} below) then 
\eqnb\label{traj_est}
\| \nabla X - \mathrm{I} \|_{L^\infty \BBB } \leq C( \| X \|_{L^\infty \BBB } ) \| v \|_{L^1 \BBB } \lec \epsilon.
\eqne
Given \eqref{traj_est}, one can use a Picard iteration to construct $X$. In particular an appropriate choice of small $\epsilon >0$ guarantees the assumed smallness of $\| \na X - I \|_{L^\infty \BB } $. 

In order to prove \eqref{traj_est}, we first note that, since $X(t) \colon \T^d \to \T^d$ is a diffeomorphism, we have
\[
\int \left( \na X (t) - I \right) dy =0
\]
for each $t\geq 0$. Thus, in light of \eqref{der_in_out}, in order to show \eqref{traj_est} it suffices to verify that
\eqnb\label{lagr_impl}
\Delta X(t,y) = \int_0^t \Delta_y v (s,X(s,y)) \d s 
\eqne
 remains small in $\B$ for all times, where $\Delta_y = \nabla_y \cdot \nabla_y$ and we denoted by $\nabla_y $ the derivative with respect to the $y$ variable. Since 
\[
\Delta_y v (t,X(t,y)) = \p_{y_i} (\p_j v \circ X \p_{i} X_j ) = \p_k \p_j v \circ X \p_i X_k \p_i X_j + \p_j v \circ X  \Delta X_j 
\]
we have
\eqnb\label{tem}
\| \Delta_y (v\circ X ) \|_{\B } \leq \| D^2 v \circ X \|_{\B } \| \na X \|_{\BB }^2 + \| \p_j v \circ X \|_{\BB } \| \Delta X_j \|_{\B }
\eqne
At this point we would like to estimate  $\| \p_j v \circ X \|_{\BB }$ by $\| \na_y (\p_j v\circ X ) \|_{\B }$, so that we could estimate it by $\| D^2 v \circ X \|_{\B } \| \na X \|_{\BB }$, that is the same as the first term on the right-hand side above. 

However, this is not immediate as the average of $\p_j v \circ X $ does not necessarily vanish (recall \eqref{der_in_out}). Instead by adding and subtracting the average we obtain
\[
\| \p_j v \circ X \|_{\BB } \lec \| D^2 v \circ X \|_{\B } \| \na X \|_{\BB } + \{ \p_j v\circ X \},
\]
and by recalling the fact that $\na = A^T \na_y $ (see \eqref{div_euler}, for example) we can estimate the average, 
\[\begin{split}
\int \p_j v (t,X(t,y)) \d y &= \int  A_{kj}\p_{y_k} v (t,X(t,y)) \d y \\
&=  \int  (A_{kj}-\delta_{kj} ) \p_{y_k} v (t,X(t,y)) \d y \\
&\leq  \| A- I \|_\infty \| \na_y (v\circ X ) \|_\infty ,
\end{split}
\]
where $A\coloneqq (\na X)^{-1}$. This gives
\eqnb\label{tempp}
\| \na v \circ X \|_{\BB } \lec \| D^2 v \circ X \|_{\B } \| \na X \|_{\BB } + \| A- I \|_{\BB } \| \na v \circ X \|_{\BB } \| \na X \|_{\BB }.
\eqne
Thus if we suppose that 
\eqnb\label{assump}
\| \na X - I \|_{\BB } \leq  1/8C\| I \|_{\BB  },
\eqne
where $C>1$ is the implicit constant in \eqref{tempp} then $\| \na X \|_{\BB } \leq 2 \| I \|_{\BB }$ and the Neumann expansion gives (as in \eqref{neumann})
\eqnb\label{assump_cons}
\| A- I \|_{\BB } \leq \sum_{k\geq 1 } \| \na X - I \|_{\BB }^k \leq \frac{1}{\| I \|_{\BB } C} \sum_{k\geq 1} 8^{-k} \leq \frac{1}{4\| I \|_{\BB }C}.
\eqne
Thus the last term in \eqref{tempp} could be absorbed by the left-hand side to give
\[
\| \na v \circ X \|_{\BB } \lec \| D^2 v \circ X \|_{\B } \| \na X \|_{\BB } .
\]
Applying this in \eqref{tem} and taking the $\| \cdot \|_{L^\infty ((0,\infty ); \B ) }$ norm of \eqref{lagr_impl} gives
\eqnb\label{Delta_in_lagr}
\begin{split}
\| \Delta X \|_{L^\infty ((0,\infty ); \B )} &\leq \|  \Delta_y (v\circ X ) \|_{L^1 \B } \\
&\leq C( \| X \|_{L^\infty  \BBB }  )  \int_0^\infty \| D^2 v \circ X \|_{\B } \\
&\leq C( \| X \|_{L^\infty \BBB }  )  \int_0^\infty \| D^2 v \|_{\B },
\end{split}
\eqne
proving the a~priori estimate \eqref{traj_est}. \\

We note that, given $X$, we also obtain (as in \eqref{assump_cons}) that 
\[
\| A- I \|_{L^\infty \BB } \lec \epsilon.
\]
Setting $u(t,y) \coloneqq v (t,X(t,y))$, $a \coloneqq \rho (t,X(t,y)) -1$ we can use \eqref{diffeo_prop} to obtain that 
\[
\| u \|_{L^1 \BBB } \lec \| v \|_{L^1 \BBB } \lec \epsilon ,
\]
by the last two lines of \eqref{Delta_in_lagr}, as well as 
\[
\| \na a \|_{L^\infty \B } = \| \na  (\rho \circ X) \|_{L^\infty \B } \lec \| \na \rho \circ X \|_{L^\infty \B } \| \na X \|_{L^\infty \BB } \lec  \epsilon 
\] 
by \eqref{traj_est} and the assumption $\| \rho-1 \|_{L^\infty \BB } \lec \epsilon$. (Recall that the diffeomorphism property \eqref{diffeo_prop} is only valid for $s\in (0,1)$.) This and the fact that $\| \{ a \} \|_{L^\infty (0,\infty )} \lec \| \rho -1 \|_{L^\infty L^\infty }\lec \epsilon $  imply that
\[
\| a \|_{L^\infty \BB } \lec \epsilon.
\]
Another application of \eqref{diffeo_prop} and the chain rule gives that 
\[
\| u_t \|_{L^1 \B } \leq \| v_t \|_{L^1 \B} + \|  \na v \|_{L^1 \BB } \| \na X \|_{L^\infty \BB } \lec \epsilon.
\]
As for $a_t $ we use the apriori estimates \eqref{cons_of_lem}, \eqref{for_step1}, \eqref{for_step2} to obtain that
\[
\| a_t \|_{L^1 \BB } \lec \epsilon,
\]
as required. This completes the proof of  Proposition~\ref{prop_main}. \\

We now prove that Proposition~\ref{prop_main} $\Rightarrow $ Theorem~\ref{thm_main}.\\

To this end, one defines $v(t,x) \coloneqq u (t,X^{-1}(x))$, $ \rho (t,x) \coloneqq 1 + a (t,X^{-1} (t,x))$, and notes that the Lagrangian trajectory $X$ satisfying $\| \na X - I \|_{L^\infty \BB }\lec \epsilon $ is already given by Proposition~\ref{prop_main}. Moreover, in the Eulerian coordinates the structure of the equations \eqref{pressureless} allows us to control the $L^1$ norm in time of $\{ v \}$. 

To be more precise we first note that $\| A - I \|_{L^\infty \BB } \lec \epsilon $, by \eqref{smallness_I_A}, and so in particular $\| \na X^{-1} - I \|_{L^\infty L^\infty } \lec \epsilon$. Thus \eqref{diffeo_prop} implies that
\[
\| \rho - 1 \|_{L^\infty \B } = \| a \circ X^{-1} \|_{L^\infty \B } \lec \| a \|_{L^\infty \B } \lec \epsilon  ,
\]
while continuity of $\rho -1 $ in time with values in $\B$ follows from the continuity of $a$, a consequence of $\| a_t \|_{L^1 \B } \lec \epsilon$.

Similarly \eqref{diffeo_prop} implies that
\eqnb\label{001}\begin{split}
\| \Delta  v \|_{L^1 \B } &= \| (\Delta_u u )\circ X^{-1} \|_{L^1 \B } \lec \| \Delta_u u \|_{L^1 \B } \\
&\lec \| \Delta u \|_{L^1 \B}+ \| A-I \|_{\BB } \left( 1+ \| A-I \|_{\BB } \right) \| \Delta u \|_{\B } \lec \epsilon ,
\end{split}
\eqne
where we also used \eqref{Delta_commutator} in the second inequality. 

We now note that
\[
\{ v \} = \int \{ v \} \rho = \int ( \{ v \} - v ) \rho ,
\]
where we used the mass conservation $\int \rho = \int \rho_0 =1$ in the first equality and, in the second equality, we used the assumption $\int \rho_0 v_0 =0$ and the momentum conservation,
\[
    \frac{\d}{\d t}\int \rho v = - \int \rho \nabla \Psi  = -\frac12 \int  \nabla |\nabla \Psi|^2  =0 ,
\]
which follows from \eqref{pressureless}, where $\Psi \coloneqq K\rho $. Thus
\eqnb\label{cons_prop}
| \{ v \} | \lec  \| \rho \|_{L^\infty } \| v- \{ v \} \|_{L^p} \lec  \| \rho \|_{L^\infty } \|\Delta  v \|_{ \B } ,
\eqne
due to \eqref{der_in_out}. This and \eqref{001} implies that $\| v \|_{L^1 \BBB } \lec \epsilon $, as required.

It remains to show the estimate for $v_t$ in $ L^1 \B$. To this end we note that $u_t = v_t \circ X + (v\circ X ) \cdot (\nabla v \circ X)$, which implies that
\[
v_t = u_t \circ X^{-1} - (u \cdot \na_u u) \circ X^{-1} .
\]
Thus \eqref{diffeo_prop} gives
\[\begin{split}
\| v_t \|_{L^1 \B   } &\leq \| u_t \circ X^{-1} \|_{L^1 \B } + \| (u \cdot \na_u u) \circ X^{-1} \|_{L^1 \B } \\
&\lec \| u_t \|_{L^1 \B } + \| u \cdot \na_u u \|_{L^1 \B }\\
& \lec \epsilon + \| u \|_{L^\infty \B } \| \na_u u\|_{L^1 \BB}\lec \epsilon (1+ \| A-I \|_{L^\infty \BB } \| \Delta u \|_{L^1 \B} )\lec \epsilon ,
\end{split}
\]
where we used \eqref{algebra_prop1} twice in the third line. 

\section*{Acknowledgements}
P.B.M. was supported by the Polish National Science Centre’s Grant No. 2018/30/M/ST1/00340 (HARMONIA). 

W.S.O. was supported in part by the Simons Foundation.

\bibliographystyle{plain}
\bibliography{literature}

\end{document}